\documentclass[a4paper,12pt]{amsart}

\usepackage{amsmath,amsthm,amssymb}
\usepackage{graphicx}

\newtheorem{corollary}{Corollary}

\newtheorem{lemma}{Lemma}

\newtheorem{proposition}{Proposition}
\newtheorem{remark}{Remark}
\newtheorem{theorem}{Theorem}

\DeclareMathOperator{\conv}{conv}
\DeclareMathOperator{\spn}{span}

\DeclareMathOperator{\pos}{pos}

\usepackage[margin=1in]{geometry}

\begin{document}
\title{Every generating polytope is strongly monotypic}
\author{Vuong Bui}
\address{Vuong Bui, Institut f\"ur Informatik, Freie Universit\"{a}t
Berlin, Takustra{\ss}e~9, 14195 Berlin, Germany \hfill\break
\mbox{\hspace{4mm}} Institute for Artificial Intelligence, Vietnam National University, Hanoi, 144 Xuan Thuy Street, Hanoi 100000, Vietnam}
\thanks{Part of the work was supported by the Deutsche Forschungsgemeinschaft
(DFG) Graduiertenkolleg ``Facets of Complexity'' (GRK 2434).  The work
was initiated during the time the author studied at Moscow Institute
of Physics and Technology.}
\email{bui.vuong@yandex.ru}

\begin{abstract}
	We prove an old conjecture of McMullen, Schneider and Shephard that every polytope with the generating property is strongly monotypic. The other direction is already known, which implies that strong monotypy and the generating property for polytopes are the same notion. A criterion for monotypic and strongly monotypic polytopes is also given.
\end{abstract}

\maketitle

\section{Introduction}
Monotypic polytopes and strongly monotypic polytopes in $\mathbb R^n$ were first introduced in \cite{mcmullen1974monotypic} by McMullen, Schneider and Shephard.  Monotypic polytopes can be seen as a subclass of simple polytopes: A polytope $P$ is monotypic if every polytope with the same set of normals as $P$ is combinatorially equivalent to $P$.

On the other hand, strongly monotypic polytopes is a subclass of monotypic polytopes: A polytope $P$ is strongly monotypic if every polytope $Q$ with the same set of normals as $P$ satisfies that the arrangements of the hyperplanes containing the facets of $P$ and $Q$ are combinatorially equivalent.

In \cite{mcmullen1974monotypic} it was shown that monotypy and strong monotypy have relations to the intersection properties. For monotypy, the intersection of every two translates of a monotypic polytope $P$ is either empty or homothetic to a Minkowski summand of $P$. For strong monotypy, we have a stronger conclusion: The intersection of every two translates of a strongly monotypic polytope $P$ is either empty or a Minkowski summand of $P$. We remind that the Minkowski sum of two sets $X,Y$ is $\{x+y: x\in X, y\in Y\}$.

Note that the above property for strongly monotypic polytopes is actually the generating property for polytopes. The original form of the generating property for a convex body $K$ is: The intersection of every family of translates of $K$ is either empty or a Minkowski summand of $K$. In \cite{karasev2001characterization}, the property was shown to be equivalent to the form where a family of translates is replaced by a pair of translates. Therefore, a strongly monotypic polytope is always a generating set. Other sets that are not necessarily polytopes were shown to posses this property in \cite{Balashov_2000}, \cite{borowska2010intersection}.

In the other direction, it is natural to ask whether every polytope with the generating property is always a strongly monotypic polytope.  It was shown to be the case for $\mathbb R^3$ in \cite{mcmullen1974monotypic}. In fact, the authors of \cite{mcmullen1974monotypic} conjectured that it holds for every dimension. In this article, we show that it is indeed the case.
\begin{theorem} \label{thm:generating-strmono}
	Every polytope with the generating property is strongly monotypic.
\end{theorem}

It follows that strong monotypy and the generating property are the same notion for polytopes. In contrast to strong monotypy, the condition for monotypy is already shown to be necessary and sufficient in \cite{mcmullen1974monotypic}:  A polytope $P$ is monotypic if and only if the intersection of every two translates of $P$ is either empty or homothetic to a Minkowski summand of $P$. It means that we only need to prove the following equivalent form of Theorem \ref{thm:generating-strmono}.
\begin{theorem} 
	\label{thm:equivalent-generating-strmono}
	Every polytope that is monotypic but not strongly monotypic does not have the generating property.
\end{theorem}

The proof of Theorem \ref{thm:generating-strmono} in Section \ref{sec:generating-strmono} is done by proving Theorem \ref{thm:equivalent-generating-strmono}. In order to do that, we need a more convenient description of monotypic and strongly monotypic polytopes. Monotypic polytopes and strongly monotopic polytopes $P$ can be recognized by the set of normals $N(P)$.\footnote{Note that the normal vectors in the article are always outward vectors that are normalized to be unit vectors by default. Sometimes it is more convenient to have them lie on a hyperplane, but it will be done explicitly. Moreover, the article deals with finite sets of normal vectors only.} In \cite{mcmullen1974monotypic}, several equivalent versions of the necessary and sufficient conditions are given. The following theorems are among them.
\begin{theorem}[Condition $M3'$ of {\cite[Theorem $1$]{mcmullen1974monotypic}}]
	\label{thm:empty-intersection}
	The monotypy of a polytope $P$ is equivalent to: If $V_1$ and $V_2$ are disjoint primitive subsets of $N(P)$ then $\pos V_1\cap\pos V_2=\{0\}$.
\end{theorem}

Here, $V$ is a \emph{primitive} subset of $N(P)$ if $V$ is linearly independent and $\pos V\cap N(P)=V$. The latter condition can be understood as the positive hull of $V$ does not contain any other normal than those in $V$. Such a situation is encountered several times throughout the text.
\begin{theorem}[Condition $S4'$ of {\cite[Theorem $2$]{mcmullen1974monotypic}}]
	\label{thm:subset-monotypic}
	The strong monotypy of a polytope $P$ is equivalent to: If $Q$ is any polytope with $N(Q)\subseteq N(P)$ then $Q$ is monotypic.
\end{theorem}

We give other equivalent conditions $D$ and $DD$ as follows.
\begin{theorem}[Condition $D$] \label{thm:monotypic-description}
		The monotypy of an $n$-dimensional polytope $P$ is equivalent to: If some $n+1$ normals of $P$ are in conical position, then their positive hull contains another normal of $P$.
\end{theorem}

In this text, a set of points is said to be separated from $0$ if there is a hyperplane strictly separating the set from $0$. Also, some points are said to be in \emph{conical position} if they are separated from $0$ and none of the points is in the positive hull of the others.

\begin{theorem}[Condition $DD$] \label{thm:strongly-monotypic-description}
		The strong monotypy of an $n$-dimensional polytope $P$ is equivalent to: Every $n+1$ normals of $P$ are not in conical position.
\end{theorem}

The equivalences in Theorem \ref{thm:monotypic-description} and Theorem \ref{thm:strongly-monotypic-description} are verified in Section \ref{sec:equivalent-characterization} using Theorem \ref{thm:empty-intersection} and Theorem \ref{thm:subset-monotypic}. Theorem \ref{thm:generating-strmono} is proved in Section \ref{sec:generating-strmono} using Theorem \ref{thm:monotypic-description} and Theorem \ref{thm:strongly-monotypic-description}.

Before closing the introduction, we illustrate in Fig. \ref{fig:monotypic-intersecting} a section of a polytope $P$ in $\mathbb R^3$ that is monotypic but not strongly monotypic (in black color) and a translate (in blue color) so that their intersection has the top most face being an edge (in red color) that is longer than the corresponding face of $P$. It means their intersection is not a Minkowski summand of $P$. In other words, $P$ does not have the generating property.

The example actually provides an idea for a proof of Theorem \ref{thm:generating-strmono} in $\mathbb R^3$, where a polytope that is monotypic but not strongly monotypic must have four normals in conical position with only one other normal in the positive hull of the four normals. The facet of this only other normal is a parallelogram. We here only sketch the approach, while the proof in Section \ref{sec:generating-strmono} provides more details with a generalization for higher dimensions, where there are some unpleasant situations that are not so nice as two parallelograms intersecting at an edge.

\begin{figure}[ht]
	\includegraphics[width=0.25\textwidth]{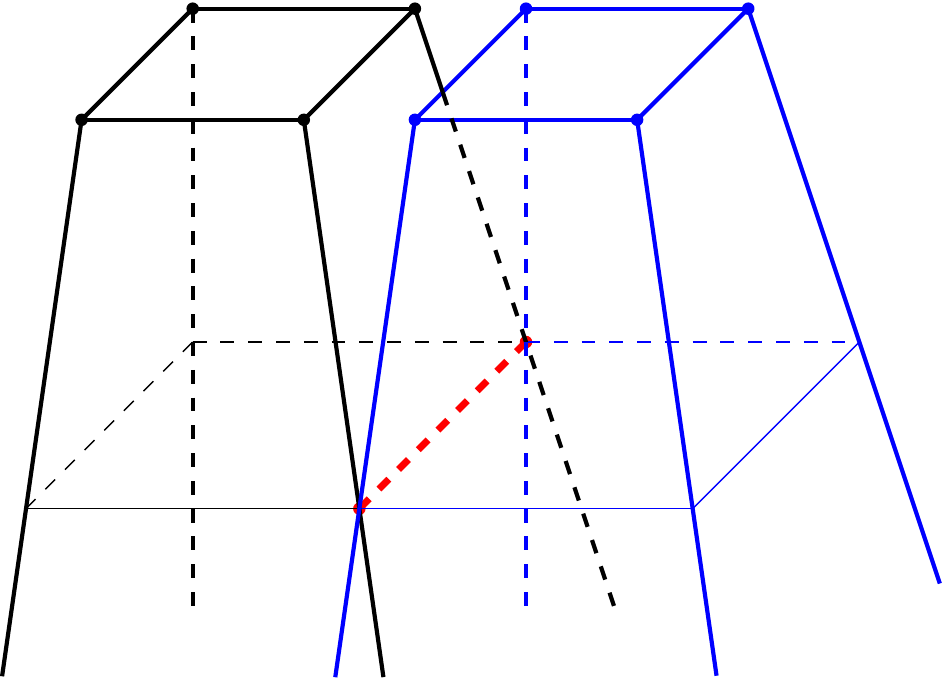}
	\caption{Part of a monotypic but not strongly monotypic polytope without the generating property}
	\label{fig:monotypic-intersecting}
\end{figure}

\section{Equivalence of the characterizations}
\label{sec:equivalent-characterization}
In order to prove Theorem
\ref{thm:monotypic-description}, we will show that the two conditions
$M3'$ and $D$ are equivalent by verifying both directions.

\begin{proposition} \label{prop:M3'->D}
	$M3'$ implies $D$.
\end{proposition}
\begin{proof}
	Suppose we do not have $D$, which means there exist $n+1$ normals in conical position with the positive hull not containing any other normal. Let $H$ be a hyperplane not through $0$ such that for each $a$ of the $n+1$ normals, ray $\overrightarrow{0a}$ cuts the hyperplane at only one point. Replacing every normal $a$ by the intersection of ray $\overrightarrow{0a}$ and the hyperplane, and applying Radon's theorem to this set on the affine space, there will be two disjoint subsets $A_1,A_2$ whose convex hulls have a nonempty intersection. Consider a point $p\in\conv A_1\cap \conv A_2$ and let $A'_i\subseteq A_i$ for $i=1,2$ be the set of vertices of a simplex whose relative interior contains $p$. In other words, the points in $A'_i$ are linearly independent. Since $A'_i$ can be seen as a subset of the normals of the polytope and its positive hull is empty of other normals, it follows that $A'_i$ is primitive. However, the positive hulls of $A'_1,A'_2$ both contain $p$, which is another point than $0$. That is we do not have $M3'$, the conclusion follows.
\end{proof}

To show that $D$ implies $M3'$, we need the following lemmas.

\begin{lemma} \label{lem:ray_as_intersection}
	Given two disjoint finite sets of normals $V_1, V_2$.  If $\pos V_1$ and $\pos V_2$ intersect at a point other than $0$, then there are $V'_1\subseteq V_1, V'_2\subseteq V_2$ such that $(\pos V'_1\cap\pos V'_2)\setminus\{0\}$ is precisely a ray in the relative interior of both the positive hulls.  Moreover, the union $V'_1\cup V'_2$ is separated from $0$.
\end{lemma}
\begin{proof}
	Let $U_1\subseteq V_1$ and $U_2\subseteq V_2$ be minimal subsets so that $\pos U_1\cap \pos U_2\ne \{0\}$.

	We show that $(\pos U_1\cap \pos U_2)\setminus\{0\}$ is a ray in the relative interior of both the positive hulls.  Suppose there are two different rays $\overrightarrow{0p},\overrightarrow{0q}$ in the intersection. 
	Let
	\[
		p=\sum_{x_i\in U_1} \lambda_i x_i=\sum_{x_i\in U_2}
		\lambda_i x_i,
	\]
	and
	\[
		q=\sum_{x_i\in U_1} \theta_i x_i=\sum_{x_i\in U_2}
		\theta_i x_i.
	\]
	(Note that there is no confusion of $\lambda_i$ for $x_i\in U_1$ and $x_i\in U_2$ since $U_1$ and $U_2$ are disjoint.  The same is for $\theta_i$.)

	Due to the minimality of $U_1, U_2$, all the coefficients $\lambda_i, \theta_i$ are positive and the points in each of $U_1,U_2$ are linearly independent (by Carath\'eodory's theorem for positive hulls). 
	Consider the point
	\[
		p-\alpha q = \sum_{x_i\in U_1} (\lambda_i-\alpha\theta_i) x_i = \sum_{x_i\in U_2} (\lambda_i-\alpha\theta_i) x_i
	\]
	where $\alpha=\min\{\lambda_i/\theta_i: x_i\in U_1\cup U_2\}$. The choice of $\alpha$ ensures that the coefficients $\lambda_i-\alpha\theta_i$ for $x_i\in U_1\cup U_2$ are all nonnegative with at least one zero, say $\lambda_{i^*}-\alpha\theta_{i^*}=0$ for $x_{i^*}\in U_1$.  Meanwhile, $p-\alpha q$ is not zero as $p,q$ are two different normals.  That $p-\alpha q\in\pos (U_1\setminus\{x_{i^*}\}) \cap \pos U_2$ contradicts the minimality of $U_1, U_2$.

	We now prove that $U_1\cup U_2$ is separated from $0$. Suppose $0\in\conv(U_1\cup U_2)$, that is
	\[
		-\sum_{x_i\in U_1}\lambda_i x_i = \sum_{x_i\in
		U_2}\lambda_i x_i
	\]
	for some nonnegative coefficients $\lambda_i$.

	Since $(\pos U_1\cap\pos U_2)\setminus\{0\}$ is a ray in the relative interior of both positive hulls, we have
	\[
		\sum_{x_i\in U_1} \theta_i x_i = \sum_{x_i\in U_2}
		\theta_i x_i
	\]
	for some positive coefficients $\theta_i$.

	Let $\alpha=\max\{\lambda_i/\theta_i: x_i\in U_1\}$, consider the equation
	\[
		\sum_{x_i\in U_1} (-\lambda_i + \alpha\theta_i) x_i =
		\sum_{x_i\in U_2} (\lambda_i + \alpha\theta_i) x_i.
	\]
	The coefficients on the right hand side are all positive while the coefficients on the left hand side are all nonnegative with at least a zero coefficient.  It means that the positive hull of a proper subset of $U_1$ intersects the positive hull of $U_2$ at a point other than $0$, contradicting the minimality of $U_1, U_2$.

	The sets $U_1,U_2$ satisfy the requirements of $V'_1,V'_2$ in the conclusion.
\end{proof}

\begin{lemma} \label{lem:subspace_causes_contradiction}
	The following claims are equivalent for a finite set $X$ that linearly spans\footnote{Throughout the text, spans are always linear spans, if not explicitly stated.} $\mathbb R^n$:

	(i) There exist $n+1$ points of $X$ in conical position with the positive hull empty of other points of $X$.

	(ii) For some $d\le n$, there exist $d+1$ points of $X$ that span a $d$-dimensional linear subspace and are in conical position with the positive hull empty of other points of $X$.
\end{lemma}
\begin{proof}
	The direction that (i) implies (ii) is trivial: Let $Y$ denote the set of the $n+1$ points and $d$ be the dimension of the subspace that $Y$ spans, we simply take some $d + 1$ points of $Y$ that span a $d$-dimensional subspace.

	We show the other direction that (ii) implies (i). At first, we start with the set $Y$ of the given $d+1$ points, whose linear span is $d$-dimensional. In each step, we increase the dimension of the span of $Y$ one by one, by adding one point to $Y$ in each step, until we have $n+1$ points. If the current number of points is less than $n+1$, the dimension is therefore less than $n$, which implies the existence of another point $p$ of $X$ not in the span of $Y$. If the positive hull of $Y\cup\{p\}$ is not empty of other points of $X$, we replace $p$ by any point of $X$ in $\pos(Y\cup\{p\})\setminus (Y\cup\{p\})$, and recursively repeat it until the positive hull is empty of other points of $X$. The process will eventually terminate as the positive hull contains fewer points of $X$ after each step and $X$ is finite. Note that the points of $Y\cup\{p\}$ are in conical position, since the points of $Y$ are already in conical position while $p$ is not in the span of $Y$. So, for any $Y$, we can increase the dimension of the linear span by one, by adding one point $p$ but still keeping the positive hull empty of other points in $X$. Therefore, we can come up with $n+1$ points satisfying the hypothesis, thus achieving (i).
\end{proof}
\begin{remark} \label{rem:subspace_causes_contradiction_extended}
	If we have (ii), we also have the same (ii) for every higher $d$, including $n$, which is itself a stronger statement than (i). Also, if we remove the phrase ``with the positive hull empty of other points of $X$'' in both claims, the lemma and the previous statement still hold with an even simpler argument. In particular, the following derivation is used in Section \ref{sec:generating-strmono}: Given a finite set $X$ that spans an $n$-dimensional space, if there are $n+1$ points of $X$ in conical position, there exist some $n+1$ points of $X$ in conical position \emph{and} span an $n$-dimensional space.
\end{remark}

Now comes the verification of $D$ implying $M3'$.
\begin{proposition} \label{prop:D->M3'}
	$D$ implies $M3'$.
\end{proposition}
\begin{proof}
	Suppose we do not have $M3'$, which means there are two disjoint primitive subsets $V_1, V_2$ of $N(P)$ such that $(\pos V_1 \cap \pos V_2)\setminus\{0\}$ is nonempty. Over all such pairs $V_1,V_2$, we consider a pair so that $\pos(V_1\cup V_2)$ is minimal (with respect to inclusion). 
	By Lemma \ref{lem:ray_as_intersection}, we can assume that $(\pos V_1\cap\pos V_2)\setminus\{0\}$ is a precisely a ray in the relative interior of both $\pos V_1, \pos V_2$, as otherwise we simply take the sets $V'_1,V'_2$ as in Lemma \ref{lem:ray_as_intersection} to replace $V_1,V_2$ (while still preserving the minimality of $\pos(V_1\cup V_2)$). Also by Lemma \ref{lem:ray_as_intersection}, we can assume that $V_1\cup V_2$ is separated from $0$.

	Consider a hyperplane that cuts every ray $\overrightarrow{0p}$ for each $p\in V_1\cup V_2$ at exactly a point $p'$. Let $U_1,U_2$ be the corresponding sets of those points $p'$. Let $U$ be the intersection of all the normals in $\pos (V_1\cup V_2)$ with the hyperplane. Note that $\conv U_1$ and $\conv U_2$ are empty of other points in $U$ (as $V_1,V_2$ are primitive).

	Although $U_1,U_2$ are contained in a hyperplane, it is more convenient to treat this affine space as an $(n-1)$-dimensional linear space by considering the \emph{translation} that takes the single point of $\conv U_1\cap \conv U_2$ to the origin. Let $T,T_1,T_2$ be the images of $U,U_1,U_2$ by the translation, respectively. 
	As $V_1,V_2$ are primitive, each of $T_1,T_2$ is the set of vertices of a simplex whose relative interior contains $0$. Moreover, as $\pos V_1\cap\pos V_2$ is a ray, the linear spaces spanned by $T_1,T_2$ are linearly independent. We assume that the spaces spanned by $T_1$ and $T_2$ are orthogonal (otherwise, we transform the space). 
	We remark that $T_1\cup T_2$ linearly (and affinely) spans a $(|T_1\cup T_2|-2)$-dimensional space.
	Also, the points in $T_1\cup T_2$ are in convex position.

	Suppose we have $D$, we can see that $T\setminus (T_1\cup T_2)$ is nonempty. Indeed, suppose otherwise that $T=T_1\cup T_2$, it follows that the points of $U_1\cup U_2$ are in conical position with the positive hull empty of other points of $U$. Meanwhile, these points span a linear space of dimension $(|T_1\cup T_2|-2)+1=|U_1\cup U_2|-1$. Now an application of Lemma \ref{lem:subspace_causes_contradiction} gives us some $n+1$ normals in conical position with the positive hull empty of other normals, contradiction to Condition $D$.  

	Let $p$ be a point in $T\setminus (T_1\cup T_2)$ that has the smallest distance to $\conv T_1$. It follows that $\conv(\{p\}\cup T_1)$ is empty of other points in $T$ since otherwise another point would be closer to $\conv T_1$ than $p$. Also note that the points of $\{p\}\cup T_1$ are in convex position since $p$ is not in the space spanned by $T_1$.

	Let $p'$ be the projection\footnote{Throughout the text, projections are always orthogonal projections.} of $p$ onto the space of $T_1$. Since $T_1$ is the set of vertices of a simplex containing $0$, there is a proper subset $T'_1\subset T_1$ so that $0\in\conv(\{p'\}\cup T'_1)$. We then have $\conv(\{p\}\cup T'_1)$ and $\conv T_2$ intersecting. Indeed, let
	\[
		\theta p'+\sum_{x_i\in T'_1} \lambda_i x_i = 0,
	\]
	where $\theta$ and $\lambda_i$ for $x_i\in T'_1$ are nonnegative numbers summing to $1$, the point
	\[
		\theta p + \sum_{x_i\in T'_1} \lambda_i x_i = \theta p''
	\]
	is then in the convex hull of $T_2$, where $p''$ is the projection of $p$ onto the space of $T_2$.

	The corresponding normals to $\{p\}\cup T'_1$ and $T_2$ form the two sets satisfying the conditions of $V_1,V_2$ but having a smaller positive hull of the union, contradiction. It follows that we do not have $D$.
\end{proof}

We now have verified Theorem \ref{thm:monotypic-description}.
\begin{proof}[Proof of Theorem \ref{thm:monotypic-description}]
		Theorem \ref{thm:monotypic-description} follows from Propositions \ref{prop:M3'->D} and \ref{prop:D->M3'}.
\end{proof}

We show the equivalence of the conditions $S4'$ and $DD$ in order to prove Theorem \ref{thm:strongly-monotypic-description}. Let us remind Condition $S4'$ for a polytope $P$ to be strongly monotypic: If $Q$ is any polytope with $N(Q)\subseteq N(P)$ then $Q$ is monotypic.

\begin{proof}[Proof of Theorem
	\ref{thm:strongly-monotypic-description}]
	In one direction, if a monotypic polytope $P$ is also strongly monotypic, it should not have $n+1$ normals in conical position (i.e.  Condition $DD$).  Indeed, suppose $V$ is the set of such $n+1$ normals, we consider the subset of $N(P)$ after removing every normal in the positive hull of $V$ except the normals in $V$ themselves, which is $N(P)\setminus((\pos V)\setminus V)$.  Any polytope taking this subset as the set of normals is not monotypic, due to the existence of the $n+1$ normals of $V$ in conical position with the positive hull not containing any other normal.

	In the other direction, if a polytope $P$ satisfies Condition $DD$, then every polytope $Q$ with $N(Q)\subseteq N(P)$ should be monotypic. Indeed, Condition $DD$ means we do not have $n+1$ normals of $N(P)$ in conical position. The same situation also applies for the subset $N(Q)\subseteq N(P)$. We have then Condition $D$ for $Q$, that is $Q$ is monotypic.
\end{proof}

\section{Equivalence of strong monotypy and the generating property}
\label{sec:generating-strmono}
By the equivalence of Theorem \ref{thm:generating-strmono} and Theorem \ref{thm:equivalent-generating-strmono}, to prove either theorem, it remains to prove that every polytope that is monotypic but not strongly monotypic does not have the generating property.

Consider such a polytope $P$, we have some $n+1$ normals of $P$ in conical position for $P$ to be not strongly monotypic (by Condition $DD$ of Theorem \ref{thm:strongly-monotypic-description}). By Remark \ref{rem:subspace_causes_contradiction_extended}, there are some $n+1$ normals of $P$ in conical position that span the whole $\mathbb R^n$. Among such collections of normals, consider a set $X$ of $n+1$ normals so that its positive hull is minimal (with respect to inclusion). Note that the positive hull must contain another normal of $P$ due to the monotypy of $P$ (by Condition $D$ of Theorem \ref{thm:monotypic-description}).

Let $Y$ be the intersection of the rays of the normals in $X$ with a hyperplane not through $0$ so that each ray cuts the hyperplane at only one point. Let $Z$ be the intersection of the rays of the normals in $N(P)$ with the hyperplane. Proposition \ref{prop:positional-normals} below shows that there is only one point $p$ of $(Z\setminus Y) \cap \conv Y$.

\begin{proposition} \label{prop:positional-normals}
	There is only one point $p$ of $(Z\setminus Y) \cap \conv Y$. Furthermore, there is a partition $Y=Y_0\sqcup Y_1\sqcup Y_2$ with a possibly empty $Y_0$ so that: If we treat the affine space of $Y$ as a linear space that takes $p$ as the origin by considering the sets $Y'=Y-p$ and $Y'_i=Y_i-p$ for $i=0,1,2$, then

	(i) The spans of $Y'_0,Y'_1,Y'_2$ are linearly independent.

	(ii) The points in $Y'_0$ are linearly independent.

	(iii) Each of $Y'_1,Y'_2$ is the set of vertices of a simplex whose relative interior contains $0$.
\end{proposition}
\begin{proof}
	Since the $n+1$ points in $Y$ affinely span an $(n-1)$-dimensional affine space, we have a partition $Y=A_1\sqcup A_2$ so that $\conv A_1\cap \conv A_2\ne\emptyset$ (by Radon's theorem). Let $p$ be a point in the intersection. We consider the linear space spanned by $Y'=Y-p$, which is $(n-1)$-dimensional since $p\in\conv Y$. Also denote $A'_1=A_1-p$ and $A'_2=A_2-p$. For each $i=1,2$, there exists a subset $Y'_i\subseteq A'_i$ so that $Y'_i$ is the set of vertices of a simplex whose relative interior contains the origin $p'=p-p=0$. Denote $Y'_0=Y'\setminus(Y'_1\cup Y'_2)$. The dimension of the span of $Y'$ is at most
	\[
		|Y'_0|+(|Y'_1|-1)+(|Y'_2|-1)=n-1.
	\]
	To make the dimension precisely $n-1$, the spans of $Y'_0,Y'_1,Y'_2$ must be linearly independent and the points in $Y'_0$ are also linearly independent.

	It remains to prove that the intersection $(Z\setminus Y)\cap\conv Y$ is precisely $p$. Suppose there are two points of $Z\setminus Y$ in the convex hull of $Y$, one of the two points, say $q$, is not $p$.
	Suppose that the spans of $Y'_0,Y'_1,Y'_2$ are orthogonal, otherwise we apply an affine transformation. Consider a representation of $q'=q-p$ in the span of $Y'$:
	\[
		q'=\sum_{x_i\in Y'} \lambda_i x_i,
	\]
	where $\lambda_i$ are all nonnegative and their sum is $1$.
	We show that the $n+1$ points of $\{q'\}\cup Y'\setminus\{x_i\}$ are in convex position for some $x_i$ and their convex hull is $(n-1)$-dimensional. Indeed, if $\lambda_i>0$ for some $x_i\in Y'_0$, then $q'$ is not in the span of the rest $Y'\setminus\{x_i\}$. In the other case that $q'\in\conv (Y'_1\cup Y'_2)$, we have either $q'\notin\conv Y'_1$ or $q'\notin\conv Y'_2$ since otherwise $q'\in\conv Y'_1\cap \conv Y'_2$ would mean $q'=0$, contradicting $q\ne p$.
	Let us say $q'\notin\conv Y'_1$, it follows that the projection $q'_2$ of $q'$ to $\conv Y'_2$ is nonzero. It means that $q'_2$ is not in the convex hull of $\{0\}\cup Y'_2\setminus \{x_i\}$ for some $x_i\in Y'_2$. Since the projection of $Y'_1$ to $\conv Y'_2$ is $0$, it follows that $q'$ is not in the convex hull of $Y'\setminus \{x_i\}$. In either case, the points $\{q'\}\cup Y'\setminus \{x_i\}$ are in convex position.
	It remains to prove that the dimension is also preserved: The dimensions of $\conv Y'$ and $\conv (Y'\setminus\{x_i\})$ for every $x_i\in Y'_1\cup Y'_2$ are the same (which is also the same as the dimension of $\conv (\{q'\}\cup Y'\setminus \{x_i\}$). Discarding any $x_i\in Y'_0$ with $\lambda_i>0$ still gives $\dim\conv Y' = \dim\conv (\{q'\}\cup Y'\setminus \{x_i\})$.
	In total, the corresponding $n+1$ points in the affine hull of $Y$ contradict the minimality of $X$.
\end{proof}

By the nature of the normals of $Y$ in Proposition \ref{prop:positional-normals}, we can assume that every normal in $Y$ forms an acute angle with $p$, otherwise we can apply an affine transformation to obtain it.
Denote by $\pi(y)$ the projection of a point $y$ to the hyperplane $\{x: \langle p,x\rangle = 0\}$. 
We show that Proposition \ref{prop:positional-normals} still holds if we replace $Y'$ and $Y'_i$ for $i=0,1,2$ by the projections $\pi(Y)$ and $\pi(Y_i)$ for $i=0,1,2$. Indeed, in the beginning of this section we consider a hyperplane so that the ray of each normal in $X$ cuts the hyperplane at precisely one point. Now instead of this hyperplane, we consider the hyperplane through $p$ and perpendicular to $p$. Since the intersection of the ray of each normal $x\in X$ to the new hyperplane is still one point due to the acuteness of the angle that each $y\in Y$ forms with $p$, the conclusion of Proposition \ref{prop:positional-normals} still holds. However, this time each point $y'=y-p$ turns out to coincide with $\pi(y)$.

Let us treat the facet $F$ with normal $p$ of $P$ as a polytope in one lower dimension by \emph{translating} the polytope $P$ so that $F$ lies in the hyperplane 
\[
	p^\perp=\{x: \langle p,x\rangle =0\}
\]
through the origin. The facets of $F$ are actually $(n-2)$-faces of $P$. In this perspective, we mention the following fact about monotypic polytopes from \cite{mcmullen1974monotypic}.
\begin{theorem} [Property $M4$ of {\cite[Theorem $1$]{mcmullen1974monotypic}}]
	\label{thm:facets-intersecting}
	For every primitive subset of $k$ normals in $N(P)$ for a monotypic polytope $P\subset\mathbb R^n$, the corresponding facets intersect at an $(n-k)$-face of $P$.
\end{theorem}
It follows from Theorem \ref{thm:facets-intersecting} that the facet of $P$ with each normal $y\in Y$ cuts $F$ at an $(n-2)$-face of $P$, since $\{y,p\}$ for every $y\in Y$ is primitive (by Proposition \ref{prop:positional-normals}). In other words, 
\[
	\pi(Y)\subseteq N(F).
\]

For convenience, we assume that the spans of $\pi(Y_0), \pi(Y_1), \pi(Y_2)$ are orthogonal (by an affine transformation).
In this way, we can represent the $(n-1)$-dimensional subspace $p^\perp$ as $p^\perp= \mathcal R_0\times \mathcal R_1\times \mathcal R_2$ where
\[
	\mathcal R_0=\spn \pi(Y_0),\qquad \mathcal R_1=\spn\pi(Y_1),\qquad \mathcal R_2=\spn\pi(Y_2).
\]
The dimensions of these subspaces $\mathcal R_0,\mathcal R_1,\mathcal R_2$ are $|Y_0|,|Y_1|-1,|Y_2|-1$, respectively.
When seeing $F$ as an $(n-1)$-polytope in $p^\perp$, the supporting halfspaces of $F$ with normals in $\pi(Y_1)$ intersect at the following set in $p^\perp$:
\[
	\mathcal R_0\times S_1\times \mathcal R_2
\]
where $S_1$ is a $(|Y_1|-1)$-simplex in $\mathcal R_1$, since $\pi(Y_1)$ is the set of vertices of a simplex whose relative interior contains $0$.
Likewise, the supporting halfspaces of $F$ with normals in $\pi(Y_2)$ intersect at 
\[
	\mathcal R_0\times \mathcal R_1 \times S_2
\]
where $S_2$ is a simplex in $\mathcal R_2$. On the other hand, since the points of $\pi(Y_0)\subset \mathcal R_0$ are linearly independent, the intersection of the supporting \emph{hyperplanes} of $F$ with normals in $\pi(Y_0)$ is 
\[
	\{x^*\}\times \mathcal R_1\times \mathcal R_2
\]
for some point $x^*\in \mathcal R_0$. It means that the supporting halfspaces of $F$ with normals in $\pi(Y_1)\cup\pi(Y_2)$ and the supporting hyperplanes of $F$ with normals in $\pi(Y_0)$ all together intersect at
\[
	G=\{x^*\}\times S_1\times S_2,
\]
which is a translate of the Cartesian product of two simplices.

Note that $S_1$ is determined completely by the supporting halfspaces of $F$ in $p^\perp$:
\begin{equation} \label{eq:horizontal-representation}
	\{x:\langle \pi(y), x\rangle \le h_F(\pi(y))\}
\end{equation}
for $y\in Y_1$. 
(The intersection of the halfspaces and $\mathcal R_1$ is $S_1$.)
The same situation also applies to $S_2$. (Here we denote $h_K(\vec{n})=\sup_{x\in K} \langle \vec{n}, x\rangle$ for a set $K$ and a vector $\vec{n}$.)

As $G$ is the intersection of supporting halfspaces and supporting hyperplanes of $F$, one may naturally ask if $G$ is a face, which is confirmed in the following proposition.
\begin{proposition} \label{prop:G-face}
	$G$ is a face of $F$.
\end{proposition}
\begin{proof}[Proof of Proposition \ref{prop:G-face}]
	In $p^\perp$ each vertex of $G$ is the intersection of the supporting hyperplanes of $F$ with normals in $\pi(Y_0)\cup\pi(Y_1\setminus\{y_1\})\cup\pi(Y_2\setminus\{y_2\})$ for some $y_1\in Y_1,y_2\in Y_2$. In the original space $\mathbb R^n$, it is the intersection of the supporting hyperplanes of $P$ with normals in $U=\{p\}\cup Y_0\cup (Y_1\setminus \{y_1\})\cup (Y_2\setminus \{y_2\})$. The set $U$ is actually primitive by Proposition \ref{prop:positional-normals}. By Theorem \ref{thm:facets-intersecting}, the facets of the normals in $U$ intersect at a face of dimension $n-|U|=0$, i.e. a vertex of $P$. 
	It follows that each vertex of $G$ is also a vertex of $P$, which is in turn a vertex of $F$, since $G\subseteq p^\perp$. In other words, $G$ is a face of $F$.
\end{proof}

Since $F$ is bounded by the hyperplanes in $p^\perp$ that are defined in \eqref{eq:horizontal-representation}, the projections $F$ onto $\mathcal R_1,\mathcal R_2$ are subsets of $S_1,S_2$. That $G$ is a face confirms that they are actually identical.
\begin{corollary} \label{cor:simplices-projections}
	$S_1,S_2$ are the projections of $F$ onto $\mathcal R_1,\mathcal R_2$, respectively.
\end{corollary}

Let $F_\epsilon$ be the intersection of $P$ and the hyperplane $\langle p,x\rangle = -\epsilon$ for some small $\epsilon\ge 0$. (We allow $\epsilon=0$, which may not sound very conventional, to let $F_0=F$, for the sake of convenience.) By the definition of monotypic polytope, $F_\epsilon$ for a small enough $\epsilon$ is combinatorially equivalent to $F$ with the same set of normals. That is $F_\epsilon$ also has a face $G_\epsilon$ corresponding to the face $G$ of $F$. 
Its projection onto $p^\perp$ has a similar representation 
\[
	\pi(G_\epsilon)=\{x^*_\epsilon\}\times S_1(\epsilon)\times S_2(\epsilon)
\]
to the representation $G=\{x^*\}\times S_1\times S_2$.
The two simplices $S_1(\epsilon), S_2(\epsilon)$ share many properties with $S_1,S_2$ (which are actually $S_1(0),S_2(0)$), e.g. the property of Corollary \ref{cor:simplices-projections} that $S_1(\epsilon),S_2(\epsilon)$ are the projections of $F_\epsilon$ onto $\mathcal R_1,\mathcal R_2$, respectively. However, the ``sizes'' of $S_i(\epsilon)$ ($i=1,2$) for different values of $\epsilon$ are different and comparable due to the acuteness of the angles, as in the following proposition.

\begin{proposition} \label{prop:compare-simplices}
	For any two small enough $\epsilon,\epsilon'$ with $\epsilon>\epsilon'$, the simplex $S_1(\epsilon')$ lies in the relative interior of the simplex $S_1(\epsilon)$. The same applies for $S_2(\epsilon')$ and $S_2(\epsilon)$.
\end{proposition}
Before presenting the proof, we observe that: Consider any hyperplane $A_q=\{x:\langle q,x\rangle=c\}$ for any normal $q$ that forms an acute angle with $p$ and some real $c$. For any two reals $d,d'$ with $d>d'$, denote $A_p=\{x:\langle p,x\rangle=d\}$ and $A'_p=\{x:\langle p,x\rangle=d'\}$, we have
\begin{equation} \label{eq:acuteness-corollary}
	h_{\pi(A_p\cap A_q)} (\pi(q)) < h_{\pi(A'_p\cap A_q)} (\pi(q)).
\end{equation}
This can be observed in the plane picture of the $2$-dimensional span of $\{p,q,\pi(q)\}$. (Note that $\pi(q)$ lies in the span of $\{p,q\}$.)

\begin{proof}[Proof of Proposition \ref{prop:compare-simplices}]
	We prove for $S_1(\epsilon')$ and $S_1(\epsilon)$ only. The claim for $S_2(\epsilon')$ and $S_2(\epsilon)$ is treated similarly.

	It follows from the representation of $S_1$ in \eqref{eq:horizontal-representation} that $S_1(\epsilon)$ is the simplex that the halfspaces 
	\[
		\{x:\langle \pi(y),x\rangle \le h_{\pi(F_\epsilon)} (\pi(y))\}
	\]
	for $y\in Y_1$ intersect in $\mathcal R_1$. Likewise, the corresponding halfspaces for $S_1(\epsilon')$ are 
	\[
		\{x:\langle \pi(y),x\rangle \le h_{\pi(F_{\epsilon'})} (\pi(y))\}
	\]
	for $y\in Y_1$.
	For each $y\in Y_1$, let $A_y$ denote the hyperplane of the facet of $P$ with normal $y$, and let $A_p, A'_p$ denote the hyperplanes of $F_\epsilon,F_{\epsilon'}$, respectively.
	Since the facet of $P$ with normal $y$ cuts $F$ at an $(n-2)$-face of $P$ (by the primitivity of $\{y,p\}$ and Theorem \ref{thm:facets-intersecting}), we have
	\[
		h_{\pi(F_\epsilon)}(\pi(y)) = h_{\pi(A_p\cap A_y)}(\pi(y)), \qquad
		h_{\pi(F_{\epsilon'})}(\pi(y)) = h_{\pi(A'_p\cap A_y)}(\pi(y)).
	\]
	The conclusion follows, as by \eqref{eq:acuteness-corollary} the acuteness of the angle that every $y\in Y_1$ forms with $p$ gives
	\[
		h_{\pi(A_p\cap A_y)}(\pi(y)) > h_{\pi(A'_p\cap A_y)}(\pi(y)).\qedhere
	\]
\end{proof}

Pick a small enough $\epsilon_0$. Let $t$ be a translation vector in $\mathcal R_2$ so that $S_2(\epsilon_0)$ and $S_2(\epsilon_0)+t$ intersect at precisely a vertex $v^*_{\epsilon_0}$ of both simplices. It follows from the representation $\pi(G_{\epsilon_0})=\{x_{\epsilon_0}^*\}\times S_1(\epsilon_0)\times S_2(\epsilon_0)$ that $G_{\epsilon_0}$ and $G_{\epsilon_0}+t$ intersect precisely at a face $H_{\epsilon_0}$ of both $G_{\epsilon_0}$ and $G_{\epsilon_0}+t$ with $\pi(H_{\epsilon_0})=\{x_{\epsilon_0}^*\}\times S_1(\epsilon_0)\times \{v_{\epsilon_0}^*\}$. The face $H_{\epsilon_0}$ is not necessarily $F_{\epsilon_0}\cap (F_{\epsilon_0}+t)$ but it must be a face of $F_{\epsilon_0}\cap (F_{\epsilon_0}+t)$, since it is the intersection of two faces $G_{\epsilon_0}, G_{\epsilon_0}+t$ of $F_{\epsilon_0}, F_{\epsilon_0}+t$, respectively.

We depict in Fig. \ref{fig:good-polytopes-intersecting} an example of translates of $G_{\epsilon_0}$ intersecting at a face. Note that there are two types of intersections since the roles of $Y_1,Y_2$ are interchangeable.\footnote{While $F$ in general stands for \emph{face}, one may think that $G$ stands for \emph{good}, and $H$ stands for \emph{horizontal}.}

\begin{figure}[ht]
	\includegraphics[width=0.35\textwidth]{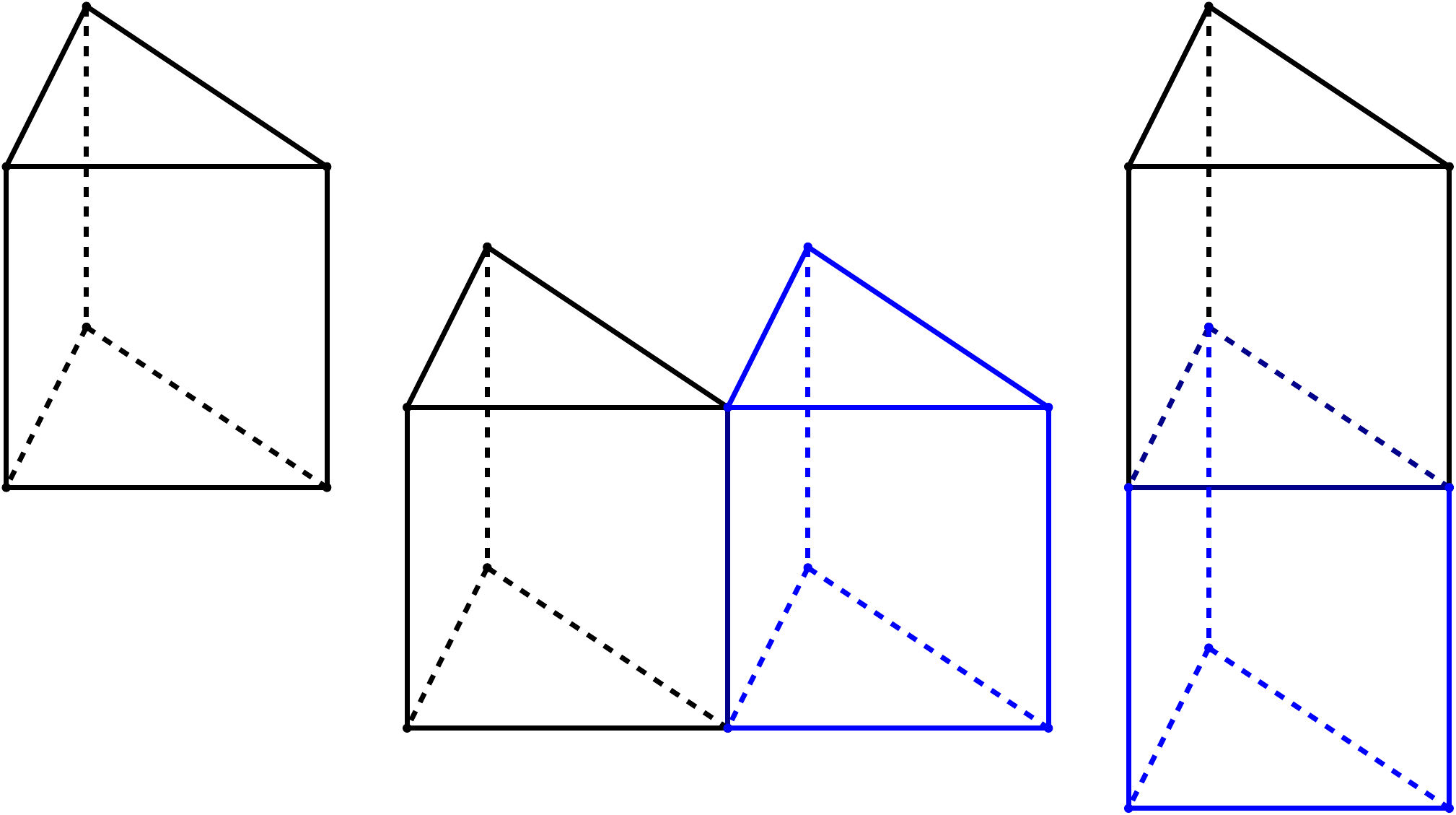}
	\caption{Translates of $G_{\epsilon_0}$ intersecting at a face in two ways}
	\label{fig:good-polytopes-intersecting}
\end{figure}

Consider any $\epsilon<\epsilon_0$. Since $S_2(\epsilon_0)\cap(S_2(\epsilon_0)+t)=\{v^*_{\epsilon_0}\}$, and $S_2(\epsilon)$ lies in the relative interior of $S_2(\epsilon_0)$ (by Proposition \ref{prop:compare-simplices}), it follows that $S_2(\epsilon)\cap(S_2(\epsilon)+t)=\emptyset$. As the projection of $F_\epsilon$ onto $\mathcal R_2$ is $S_2(\epsilon)$ (by Corollary \ref{cor:simplices-projections}), the projection of $F_\epsilon+t$ onto $\mathcal R_2$ is $S_2(\epsilon)+t$. It follows that $F_\epsilon\cap (F_\epsilon + t)=\emptyset$. As $\epsilon$ can be any value smaller than $\epsilon_0$, the intersection $F_{\epsilon_0}\cap (F_{\epsilon_0}+t)$ is the face with normal $p$ of $P\cap (P+t)$. It means that $H_{\epsilon_0}$ is also a face of $P\cap (P+t)$. 
Note that $H_{\epsilon_0}$ is a translate of $S_1(\epsilon_0)$. Meanwhile, the corresponding face $H$ of $F$ is a translate of $S_1$. It follows from Proposition \ref{prop:compare-simplices} with $S_1=S_1(0)$ that the relative interior of $H_{\epsilon_0}$ contains a translate of $H$, thus $H_{\epsilon_0}$ is not a Minkowski summand of $H$ (note that the simplices have positive dimensions).
Therefore, $P\cap (P+t)$ cannot be a Minkowski summand of $P$. In other words, $P$ does not have the generating property.

\section*{Acknowledgement}
The author would like to thank Roman Karasev for his patient reading and commenting on various pieces in early drafts. The author is also grateful for several interesting discussions with Peter McMullen.

\bibliographystyle{unsrt}
\bibliography{genstrmono}

\end{document}